\documentclass[oneside,reqno]{amsart}
\usepackage{etex}
\usepackage[a4paper]{geometry}
\usepackage{amssymb,amsfonts,amsmath}
\usepackage{comment} 
\usepackage[all,arc]{xy}
\usepackage{enumerate}
\usepackage{mathrsfs,mathtools}
\usepackage{todonotes,booktabs}
\usepackage{stmaryrd}
\usepackage{marvosym}
\usepackage{graphicx}
\usepackage{pdflscape}

\usepackage{bbm}
\usepackage{tikz}
\usepackage{tikz-cd}
\usetikzlibrary{trees}
\usetikzlibrary[shapes]
\usetikzlibrary[arrows]
\usetikzlibrary{patterns}
\usetikzlibrary{fadings}
\usetikzlibrary{backgrounds}
\usetikzlibrary{decorations.pathreplacing}
\usetikzlibrary{decorations.pathmorphing}
\usetikzlibrary{positioning}
	
\def\biblio{\bibliography{duality}\bibliographystyle{alpha}}

\usepackage{xcolor} 
\usepackage{graphicx}

\usepackage[pagebackref]{hyperref}

\definecolor{dark-red}{rgb}{0.5,0.15,0.15}
\definecolor{dark-blue}{rgb}{0.15,0.15,0.6}
\definecolor{dark-green}{rgb}{0.15,0.6,0.15}
\hypersetup{
    colorlinks, linkcolor=dark-red,
    citecolor=dark-blue, urlcolor=dark-green
}

\renewcommand*{\backref}[1]{}
\renewcommand*{\backrefalt}[4]{%
  \ifcase #1 %
No citations.
  \or
(cit. on p. #2).%
  \else
(cit on pp. #2).%
  \fi%
}
\usepackage{subfiles}

\usepackage[nameinlink,capitalise,noabbrev]{cleveref}

\newtheorem{thm}{Theorem}[section]
\newtheorem{cor}[thm]{Corollary}
\newtheorem{prop}[thm]{Proposition}
\newtheorem{lem}[thm]{Lemma}

\theoremstyle{definition}
\newtheorem{defn}[thm]{Definition}

\newtheorem{ex}[thm]{Example}

\theoremstyle{remark}
\newtheorem{rem}[thm]{Remark}

\theoremstyle{theorem}
\newtheorem*{thm*}{Theorem}

\newtheorem*{prop*}{Proposition}
\newtheorem*{cor*}{Proposition}

\newtheorem*{thm1*}{Theorem 4.7}
\newtheorem*{cor1*}{Corollary 3.3}
\newtheorem*{prop1*}{Proposition 5.3}

\makeatletter
\let\c@equation\c@thm
\makeatother
\numberwithin{equation}{section}

\DeclareMathOperator{\Sp}{Sp}
\DeclareMathOperator{\Hom}{Hom}

\DeclareMathOperator{\cC}{\mathcal{C}}

\DeclareMathOperator{\cN}{\mathcal{N}}

\DeclareMathOperator{\Ext}{Ext}
\DeclareMathOperator{\Tor}{Tor}

\DeclareMathOperator{\Mod}{Mod}

\newcommand{\N}{\mathbb{N}}
\newcommand{\Q}{\mathbb{Q}}

\DeclareMathOperator{\cof}{cof}

\DeclareMathOperator{\Sym}{Sym}
\DeclareMathOperator{\sm}{\wedge}

\newcommand{\Z}{\mathbb{Z}}

\Crefname{figure}{Figure}{Figures}
\Crefname{assu}{Assumption}{Assumptions}
\Crefname{lem}{Lemma}{Lemmas}
\Crefname{thm}{Theorem}{Theorems}

\newcommand{\recollement}[5]{
\xymatrix{{#1} \ar[r]|-{#2} & #3 \ar[r]|-{#4} \ar@<1ex>[l]^-{{#2}_!} \ar@<-1ex>[l]_-{{#2}^*} & #5, \ar@<1ex>[l]^-{{#4}!} \ar@<-1ex>[l]_-{{#4}^*}
}}
\let\lim\relax

\DeclareMathOperator{\lim}{lim}

\newcommand{\F}{\mathbb{F}}

\newcommand{\pc}[1]{{#1}_{p}^{\wedge}}

\title{On the comparison of stable and unstable $p$-completion}
\author{Tobias Barthel}
\address{Department of Mathematical Sciences, University of Copenhagen, Universitetsparken 5, 2100 K{\o}benhavn {\O}, Denmark}
\email{tbarthel@math.ku.dk}

\author{A.~K.~Bousfield}
\address{Department of Mathematics, Statistics and Computer Sciences, University of Illinois at Chicago, 851 S.~Morgan St.~(M/C 249), Chicago, IL 60607-7045, USA}
\email{bous@uic.edu}

\date{\today}

\begin{document}

\begin{abstract}
In this note we show that a $p$-complete nilpotent space $X$ has a $p$-complete suspension spectrum if and only if its homotopy groups $\pi_*X$ are bounded $p$-torsion. In contrast, if $\pi_*X$ is not all bounded $p$-torsion, we locate uncountable rational vector spaces in the integral homology and in the stable homotopy groups of $X$. To prove this, we establish a homological criterion for $p$-completeness of connective spectra. Moreover, we illustrate our results by studying the stable homotopy groups of $K(\Z_p,n)$ via Goodwillie calculus.
\end{abstract}

\maketitle

\setcounter{tocdepth}{1}
\def\biblio{}

\section{Introduction}

The notion of $p$-completion plays a fundamental role in algebra and topology, for it provides effective means to isolate and study $p$-primary properties. Applied to homotopy theory by Bousfield and Kan~\cite{bousfield_kan} as well as Sullivan~\cite{sullivan_genetics} and developed further in~\cite{bousfield_locspaces,bousfield_locspectra}, it has since become one of the standard tools in the hands of algebraic topologists. However, there appears to be no general account of the comparison between unstable and stable $p$-completion in the literature, which is the question we address in the present note.

Our main goal is to characterize $p$-complete spaces which have $p$-complete suspension spectra:
\begin{thm1*}
If $X$ is a $p$-complete nilpotent space, then $\Sigma^{\infty}X$ is $p$-complete if and only if $\pi_nX$ is bounded $p$-torsion for each $n$. 
\end{thm1*}

In fact, we exhibit a sharp dichotomy of $p$-complete nilpotent spaces: if $X$ is a $p$-complete nilpotent space whose homotopy groups are not all bounded $p$-torsion, then the integral homology groups and stable homotopy groups of $X$ both contain an uncountable rational vector space. As a consequence, we deduce that a nilpotent space $X$ with derived $p$-complete integral homology and unstable homotopy must have both $H_n(X;\Z)$ and $\pi_nX$ of bounded $p$-torsion for all $n$.

In a first step towards the proof of the theorem, we complement the second author's characterization of $p$-complete spectra in terms of homotopy groups with an integral homological criterion, using a mild generalization of Serre classes appropriate for stable homotopy theory. This is in sharp contrast to the aforementioned fact that the integral homology of $p$-complete spaces is not well-behaved, and thus cannot be used to characterize $p$-completeness of spaces.

\begin{cor1*}
A bounded below spectrum $X$ is $p$-complete if and only if $H_*(X;\Z)$ is derived $p$-complete in each degree. 
\end{cor1*}

In order to use this result to prove the theorem, we need to detect rational classes in the homology of $p$-complete spaces whose homotopy is not bounded $p$-torsion. This rests on the study of the integral homology of $p$-complete spheres. 
We end this note with a sample computation, illustrating how Goodwillie calculus allows us to detect rational classes in the stable homotopy groups of the Eilenberg--MacLane space $K(\Z_p,n)$.

\begin{prop1*}
For $n \ge 1$ and $k> 1$, the stable homotopy group $\pi_{nk}\Sigma^{\infty}K(\Z_p,n)$ contains an uncountable rational vector space. In particular, $\Sigma^{\infty}K(\Z_p,n)$ is not $p$-complete.
\end{prop1*}

In fact, we also give a short alternative argument based on the integral homology of $K(\Z_p,n)$.

\subsection*{Conventions}

Throughout this paper, $p$ will be a fixed prime number and $\Z_p$ denotes the $p$-adic integers. We say that a nilpotent group $N$ is bounded $p$-torsion if there exists an $m$ such that for all $x \in N$, we have $x^{p^m} = 1$. A graded nilpotent group $N_*$ is said to be of bounded $p$-torsion if $N_k$ is bounded $p$-torsion for each $k$; however, we do not require a uniform bound. Whenever we are in a graded context, we indicate the degree of an abelian group $A$ by square brackets, i.e., $A[n]$ refers to $A$ placed in degree $n$. If $X$ is a topological space, then $H_*(X;A)$ is the reduced homology of $X$ with coefficients in $A$. For a space or spectrum $X$, we write $\tau_{\le n}X = \tau_{< n+1}X$ for the $n$-th Postnikov section of $X$ and $\tau_{\ge n+1}X = \tau_{>n}X$ for the fiber of the canonical map $X \to \tau_{\le n}X$.

\subsection*{Acknowledgements}

We are grateful to Peter May for suggesting the authors get in touch over this problem. Furthermore, the first author would like to thank Bj{\o}rn Dundas, Frank Gounelas, Jesper Grodal, and Thomas Nikolaus for helpful conversations about $p$-completion, and has been partially supported by the DNRF92.

\section{Preliminaries on $p$-completion}

We briefly recall the basic properties of $p$-completion for nilpotent groups, topological spaces, and spectra. With the exceptions of \Cref{lem:boundedtors} and \Cref{prop:nilboundedtors}, this material is mostly taken from \cite{bousfield_kan,bousfield_locspaces,bousfield_locspectra}, and we refer to these sources as well as \cite{hovey_morava_1999,mayponto} for further references. 

\subsection{Algebraic $p$-completion for abelian groups}

In general, the $p$-completion functor $M \mapsto \lim_iM/p^iM$ on the category of abelian groups is neither left nor right exact, so one studies its zeroth and first left derived functors $L_0$ and $L_1$, respectively. An abelian group $M$ is called derived $p$-complete (or $\Ext$-$p$-complete or $L$-complete) if the natural completion map $M \to L_0M$ is an isomorphism. For each abelian group $M$, the map $M \to L_0M$ will then be the universal homomorphism from $M$ to a derived $p$-complete abelian group by \cite[Ch.~VI, 3.2]{bousfield_kan}. We will denote the full subcategory of derived $p$-complete abelian groups by $\cC_p$. 

\begin{prop}\label{prop:pcompletemodules}
The category $\cC_p$ is a full abelian subcategory of $\Mod_{\Z}$ closed under extensions and limits. 
Furthermore, for any $M \in \Mod_{\Z}$ there is a short exact sequence
\[
\xymatrix{0 \ar[r] & \lim_i^1\Hom_{\Z}(\Z/p^i,M) \ar[r] & L_0M \ar[r] & \lim_iM/p^iM \ar[r] & 0}
\]
relating derived $p$-completion to ordinary $p$-completion.
\end{prop}
\begin{proof}
This is essentially proven in \cite[Ch.~VI, 2.1]{bousfield_kan}, but can also be deduced as a special case of \cite[Thms.~A.2 and A.6]{hovey_morava_1999}.
\end{proof}

We will later make use of the following observation. 
\begin{lem}\label{lem:boundedtors}
If $A \in \cC_p$ is torsion, then $A$ is bounded $p$-torsion.
\end{lem}
\begin{proof}
We give two proofs, a conceptual one and an elementary argument. First, any derived $p$-complete group $A$ is cotorsion, so the Baer--Fomin theorem~\cite{baer_boundedtorsion} implies that $A$ is the direct sum of a divisible group and a bounded torsion group. Since $A$ is derived $p$-complete, it must be reduced, hence the divisible summand is trivial. 

Second, suppose that the conclusion of the lemma is false, i.e., that there exists a sequence $(a_i)_{i\in \N}$ of elements of $A$ such that the order of $a_i$ is $p^i$. Set $x_j= \sum_{i=0}^{j-1}a_{2i+1}p^i$, then the element $x = (x_1,x_2,x_3,\ldots)\in\prod_{j\in \N}A$ lies in $\lim_{j}A/p^j$. By construction, $x$ is not $p$-torsion, which contradicts the fact that $A \to \lim_{j}A/p^j$ is surjective, forcing $\lim_{j}A/p^j$ to be $p$-torsion.
\end{proof}

\begin{rem}
By a theorem of Pr{\"u}fer, the conclusion of the lemma implies that $A$ must in fact be a direct sum of cyclic $p$-groups. 
\end{rem}

\subsection{Algebraic $p$-completion for nilpotent groups}

Recall from \cite[Ch.~VI, \S2]{bousfield_kan} that the notion of derived $p$-completion can be extended to nilpotent groups, as follows: If $\pc{X}$ denotes the Bousfield--Kan $p$-completion of a nilpotent space $X$ as recalled in the next subsection, then we define the derived $p$-completion of the nilpotent group $N$ as $L_0N = \pi_1(\pc{K(N,1)})$ and $L_1N = \pi_2(\pc{K(N,1)})$. A nilpotent group $N$ is called derived $p$-complete if the completion map $N \to L_0N$ is an isomorphism; for each nilpotent group $N$, the map $N \to L_0N$ will then be the universal homomorphism from $N$ to a derived $p$-complete nilpotent group by \cite[Ch.~VI, 3.2]{bousfield_kan}. We denote the category of derived $p$-complete nilpotent groups by $\cN_p$. 

The inclusion functor $\cC_p \to \cN_p$ has a left adjoint given by taking a derived $p$-complete nilpotent group $N$ to the derived $p$-completion of its abelianization $L_0(N/[N,N])$. Note that the unit of this adjunction is surjective, i.e., for any derived $p$-complete nilpotent group $N$, the canonical map $N \to L_0(N/[N,N])$ is surjective. Indeed, since $L_0$ preserves epimorphisms of nilpotent groups, all maps in the following commutative diagram are surjective:
\[
\xymatrix{N \ar[r] \ar[d]_{\cong} & N/[N,N] \ar[d] \\
L_0N \ar[r] & L_0(N/[N,N]).}
\]
We obtain the following generalization of \Cref{lem:boundedtors}:

\begin{prop}\label{prop:nilboundedtors}
The following conditions are equivalent for $N \in \cN_p$:
	\begin{enumerate}
		\item $N$ is torsion.
		\item $L_0(N/[N,N])$ is torsion.
		\item $N$ is bounded $p$-torsion. 
	\end{enumerate}
\end{prop}
\begin{proof}
The surjectivity of the map $N \to L_0(N/[N,N])$ observed above immediately gives the implication (1) $\Rightarrow$ (2), while (3) $\Rightarrow$ (1) is trivial.  

Assume that $L_0(N/[N,N])$ is torsion and thus bounded $p$-torsion by \Cref{lem:boundedtors}. Consider the lower central series of $N$,
\[
N = \gamma_1N \supseteq \gamma_2N \supseteq \ldots \supseteq \gamma_mN = 1,
\]
with successive abelian quotients $Q_i(N) = \gamma_iN/\gamma_{i+1}N$. We claim that, for each $i\ge 1$, $Q_i(N)$ is a direct sum of a $p$-divisible group and a bounded $p$-torsion group. Indeed, we start with the abelianization $Q_1(N) = N/[N,N]$ of $N$. Lemma 3.7 in \cite[Ch.~VI]{bousfield_kan} implies that the kernel of the completion map $Q_1(N) \to L_0Q_1(N)$ is $p$-divisible, so the claim holds for $Q_1(N)$. The general case follows from this, because $\bigoplus_{i\ge 1}Q_i(N)$ is generated as a Lie algebra by $Q_1(N)$. By \cite[Ch.~VI, 2.5]{bousfield_kan}, there is an exact sequence
\[
\xymatrix{L_0Q_i(N) \ar[r] & L_0(N/\gamma_{i+1}N) \ar[r] & L_0(N/\gamma_{i}N) \ar[r] & 1}
\]
for any $i \ge 1$. Using the previous claim, $L_0Q_i(N)$ is bounded $p$-torsion, so we see inductively that $L_0(N/\gamma_iN)$ is bounded $p$-torsion for all $i\ge 1$, hence (3) holds. 
\end{proof}

\begin{rem}
The implication (1) $\Rightarrow$ (3) in the previous proposition could also be proven more directly via the upper central series of $N$, whose quotients are known to be derived $p$-complete by \cite[VI.~3.4(ii)]{bousfield_kan}, but this result would be insufficient for our later use. 
\end{rem}

\subsection{Topological $p$-completion}

In \cite{bousfield_kan}, Bousfield and Kan introduced the notion of $p$-completion for topological spaces, lifting the algebraic notion defined above to topology. In general, the $p$-completion of a space is difficult to describe, but the theory simplifies significantly for nilpotent spaces; in particular, in this case $p$-completion coincides with $H\F_p$-localization \cite{bousfield_locspaces}. Furthermore, for nilpotent spaces with $\F_p$-homology of finite type, $p$-completion can be identified with $p$-profinite completion due to Sullivan~\cite{sullivan_genetics}. Similarly, the category of spectra admits (at least) two notions of $p$-completion, given either by $H\F_p$-localization or, the one we will use here, localization at the mod $p$ Moore spectrum $S^0/p$, see~\cite{bousfield_locspectra}. The next result summarizes the relation between these constructions and lists their basic properties. 

\begin{thm}[Bousfield, Kan]\label{thm:bkpcompletion}\
\begin{enumerate}
	\item A nilpotent space $X$ is $p$-complete if and only if $\pi_nX$ is derived $p$-complete for all $n \in \N$. Moreover, the notions of $p$-completion and $H\F_p$-localization coincide for nilpotent spaces.
	\item A spectrum $X$ is $p$-complete if and only if $\pi_nX$ is derived $p$-complete for all $n \in \Z$. If $X$ is bounded below, then $X$ is $p$-complete if and only if $X$ is $H\F_p$-local.
\end{enumerate}
Moreover, if $X$ is a nilpotent space or spectrum, then there exists a short exact sequence computing the unstable or stable homotopy groups of its $p$-completion, respectively:
\[
\xymatrix{0 \ar[r] & L_0\pi_nX \ar[r] & \pi_n(\pc{X}) \ar[r] & L_1\pi_{n-1}X \ar[r] & 0}
\]
for any $n$, where $L_i(-) \cong \Ext_{\Z}^{1-i}(\Z/p^{\infty},-)$ are the derived functors of $p$-completion.
\end{thm}

\section{Generalized Serre theory}

The full subcategory  $\cC_p$ of $\Mod_{\Z}$ is not closed under subobjects or quotients, and thus does not form a Serre class in the usual sense. This necessitates a mild generalization of Serre's mod $\cC$ theory which we develop in this section. 

\begin{defn}
A weak Serre class is a full subcategory $\cC \subseteq \Mod_{\Z}$ such that if 
\[
\xymatrix{A_1 \ar[r] & A_2 \ar[r] & A_3 \ar[r] & A_4 \ar[r] & A_5}
\]
is an exact sequence in $\Mod_{\Z}$ with $A_1,A_2,A_4,A_5 \in \cC$, then also $A_3 \in \cC$. 
\end{defn}

More explicitly, this means that $\cC \subseteq \Mod_{\Z}$ is a full additive subcategory closed under kernels, cokernels, and extensions. It follows that $\cC$ is also closed under tensoring and $\Tor_1^{\Z}$ with respect to finitely generated abelian groups. For instance, any Serre subcategory of $\Mod_{\Z}$ is a weak Serre class, but the converse does not hold. The main example of interest to us here is the category $\cC_p$ of derived $p$-complete abelian groups, see \Cref{prop:pcompletemodules}. 

\begin{prop}\label{prop:genserretheory}
Suppose $\cC$ is a weak Serre class. If $X$ is a bounded below spectrum, then the following two conditions are equivalent:
	\begin{enumerate}
		\item $\pi_nX \in \cC$ for all $n \in \Z$.
		\item $H_n(X;\Z) \in \cC$ for all $n \in \Z$. 
	\end{enumerate}
\end{prop}
\begin{proof}
Assume the first condition holds; we will argue via the Postnikov tower $(\tau_{\le n}X)$ of $X$. For simplicity, we will write $H_*(Y)$ for the integral homology of a spectrum $Y$ throughout this proof.

To start with, we need to show that $H_*(HA) \in \cC$ for $A \in \cC$. Using the isomorphisms $H_*(HA) \cong H_*(H\Z;A)$, the universal coefficient theorem gives a short exact sequence
\[
\xymatrix{0 \ar[r] & H_*(H\Z) \otimes_{\Z} A \ar[r] & H_*(HA) \ar[r] & \Tor_1^{\Z}(H_{*-1}(H\Z),A) \ar[r] & 0.}
\]
In each degree, the integral Steenrod algebra $H_*(H\Z)$ is finitely generated over $\Z$, as follows from Serre theory for the class of finitely generated abelian groups. Therefore, the outer terms of this sequence are in $\cC$. This shows $H_*(HA) \in \cC$ as well. 

Given $n\in\Z$, we will now prove that $H_n(X) \in \cC$. Since $H_n(\tau_{>n}X) = 0 = H_{n-1}(\tau_{>n}X)$ by connectivity, we see that $H_n(X) \cong H_n(\tau_{\le n}X)$. This reduces the claim to proving that $H_*(\tau_{\le n}X) \in \cC$. This follows inductively, using the exact sequence
\[
\resizebox{\textwidth}{!}{
\xymatrix{H_{*+1}(\tau_{\le n-1}X) \ar[r] & H_{*}(\Sigma^nH\pi_nX) \ar[r] & H_*(\tau_{\le n}X) \ar[r] & H_*(\tau_{\le n-1}X) \ar[r] & H_{*-1}(\Sigma^nH\pi_nX)}}
\]
associated to the fiber sequence $\Sigma^nH(\pi_nX) \to \tau_{\le n}X \to \tau_{\le n-1}X$. Since $H_k(H\pi_nX) \in \cC$ for all $k\in\Z$, this gives the implication $(1) \Rightarrow (2)$. 

For the converse, consider the convergent Atiyah--Hirzebruch spectral sequence
\[
E_{s,t}^2 \cong H_s(X;\pi_tS^0) \implies \pi_{s+t}X.
\]
Since $\pi_tS^0$ is finitely generated over $\Z$ for each $t\in \Z$, $H_s(X;\pi_tS^0) \in \cC$ for each bidegree $(s,t)$, hence $\pi_nX$ is also in $\cC$ for all $n\in \Z$. 
\end{proof}

When applied to the weak Serre class $\cC_p$, we obtain a homological characterization of $p$-completeness for bounded below spectra. 

\begin{cor}\label{cor:homcharpcompletion}
For a bounded below spectrum $X$, the following conditions are equivalent:
	\begin{enumerate}
		\item	$X$ is $p$-complete.
		\item $\pi_nX$ is derived $p$-complete for all $n$.
		\item $H_n(X;\Z)$ is derived $p$-complete for all $n$.
	\end{enumerate}
\end{cor}
\begin{proof}
The equivalence of (1) and (2) is the content of \Cref{thm:bkpcompletion}(2), while (2) is equivalent to (3) by \Cref{prop:genserretheory}.
\end{proof}

We deduce that the integral homology of $p$-complete spaces is well-behaved in the stable range. 

\begin{cor}\label{cor:stablerange}
Suppose $X$ is $p$-complete space. If $X$ is $n$-connected, then $H_k(X;\Z)$ is derived $p$-complete for all $k\le 2n$. 
\end{cor}
\begin{proof}
Since $\pi_k\Sigma^{\infty}X \cong \pi_kX$ for $k\le 2n$ by the Freudenthal suspension theorem, \Cref{thm:bkpcompletion} implies that $\pi_*\tau_{\le 2n}\Sigma^{\infty}X$ is derived $p$-complete in each degree, hence so is $H_*(\tau_{\le 2n}\Sigma^{\infty}X;\Z)$ by \Cref{cor:homcharpcompletion}. We thus get that $H_k(X;\Z)  \cong H_k(\Sigma^{\infty}X;\Z) \cong H_k(\tau_{\le 2n}\Sigma^{\infty}X;\Z)$ is derived $p$-complete for $k\le 2n$. 
\end{proof}

\section{The comparison}

In this section, we first study the relation between $p$-completion for spectra and spaces under the infinite loop space functor $\Omega^{\infty}$, and then prove our main theorem. 

\subsection{Infinite loop spaces}

It is easy to deduce from \Cref{thm:bkpcompletion} the following relation between unstable and stable $p$-completion under $\Omega^{\infty}$. 

\begin{prop}
For $0$-connected spectra $X$ and $Y$, we have:
	\begin{enumerate}
		\item $X$ is $p$-complete if and only if $\Omega^{\infty}X$ is $p$-complete.
		\item A map $f\colon X \to Y$ is an $H\F_p$-equivalence if and only if $\Omega^{\infty}f$ is an $H\F_p$-equivalence.
		\item The canonical comparison map $\pc{(\Omega^{\infty}X)} \to \Omega^{\infty}(\pc{X})$ is an equivalence. 
	\end{enumerate}
\end{prop}
\begin{proof}
Since $\pi_*\Omega^{\infty}X \cong \pi_*X$ and $\Omega^{\infty}X$ is nilpotent, the first claim is a direct consequence of \Cref{thm:bkpcompletion}. In order to prove (2), note that $f$ is an $H\F_p$-equivalence if and only if the homotopy groups $\pi_*\cof(f)$ of the cofiber of $f$ are uniquely $p$-divisible. This is equivalent to the statement that the $\F_p$-homology $H_*(\Omega^{\infty}\cof(f);\F_p)$ is trivial. The Serre spectral sequence associated to the fiber sequence
\[
\xymatrix{\Omega^{\infty}X \ar[r]^-{\Omega^{\infty}f} & \Omega^{\infty}Y \ar[r] & \Omega^{\infty}\cof(f)}
\]
thus shows that this happens if and only if $\Omega^{\infty}f$ is an $H\F_p$-equivalence. 

Statement (1) implies that $\Omega^{\infty}(\pc{X})$ is $p$-complete, so the map $\Omega^{\infty}(X) \to \Omega^{\infty}(\pc{X})$ factors canonically through $\phi\colon\pc{(\Omega^{\infty}X)} \to \Omega^{\infty}(\pc{X})$, making the following diagram commute:
\[
\xymatrix{\Omega^{\infty}X \ar[r] \ar[rd] & \pc{(\Omega^{\infty}X)} \ar[d] \\
& \Omega^{\infty}(\pc{X}).}
\]
By Statement (2), both the horizontal and the diagonal map are $H\F_p$-equivalences, hence so is the vertical comparison map.
\end{proof}

\begin{rem}
Let $\Omega_0^{\infty}$ be the $0$-component of $\Omega^{\infty}$. The last part of the proposition can be strengthened to an equivalence $\pc{(\Omega_0^{\infty}X)} \to \Omega_0^{\infty}(\pc{X})$ for any connective spectrum $X$ such that $\pi_0X$ does not contain any copies of $\Z/p^{\infty}$. To prove this directly, one may use the short exact sequences displayed at the end of \Cref{thm:bkpcompletion}. 
\end{rem}

\subsection{Suspension spectra}

We now turn to the comparison under $\Sigma^{\infty}$. In odd dimensions, the next result has also been observed in \cite[Rem.~VI.5.7]{bousfield_kan}, see also \cite[Rem.~11.1.5]{mayponto}.

\begin{lem}\label{lem:rationalclass}
Let $n\ge 1$ and write $S_p^n$ for the $p$-completion of $S^n$. There exists an uncountable rational vector space in $H_{2n}(S_p^n;\Z)$ which injects into $H_{2n}(K(\Z_p,n);\Z)$ under the  map $S_p^n \to \tau_{\le n}S_p^n \simeq K(\Z_p,n)$. 
\end{lem}
\begin{proof}
Consider the following segment of the Serre long exact sequence for the fibration $F \to S_p^n \to K(\Z_p,n)$:
\[
\xymatrix{H_{2n}(F;\Z) \ar[r] & H_{2n}(S_p^n;\Z) \ar[r] & H_{2n}(K(\Z_p,n);\Z) \ar[r] & H_{2n-1}(F;\Z) \ar[r] & \ldots.}
\]
\Cref{cor:stablerange} implies that $H_{2n}(F;\Z)$ and $H_{2n-1}(F;\Z)$ are derived $p$-complete. Recalling that $\Hom_{\Z}(\Q,A)=0=\Ext_{\Z}^1(\Q,A)$ whenever $A$ is derived $p$-complete, we see that the natural map  
$\Hom_{\Z}(\Q,H_{2n}(S_p^n;\Z)) \to \Hom_{\Z}(\Q,H_{2n}(K(\Z_p,n);\Z))$ is surjective. 
Thus, it will suffice to show that $H_{2n}(K(\Z_p,n);\Z)$ contains an uncountable rational vector space, which will be verified in the homological proof of \Cref{prop:emspaces} below. 
\end{proof}

Note that, because $H_*(S_p^n; \F_p) \cong H_*(S^n; \F_p) \cong \F_p[n]$, an application of the universal coefficient theorem shows that $H_k(S_p^n; \Z)$ is rational for all $k>n$.

\begin{lem}\label{lem:detection}
Suppose $N$ is a derived $p$-complete nilpotent (abelian) group and $n=1$ ($n\ge 1$). If  $N$ is not bounded $p$-torsion, then there exists an element $x \in N$ of infinite order inducing a monomorphism $H_*(K(\Z_p,n);\Q) \to H_*(K(N,n);\Q)$.
\end{lem}
\begin{proof}
By assumption on $N$ and \Cref{prop:nilboundedtors}, $L_0(N/[N,N])$ contains elements of infinite order. Let $\overline{x}$ be such an element and let $x \in N$ be a lift of $\overline{x}$. For the remainder of the proof we assume $n=1$; the (easier) case $n\ge 2$ and $N$ abelian is proven similarly. The element $x$ induces a map
\[
\xymatrix{K(\Z_p,1) \ar[r] & K(N,1) \ar[r] & K(L_0(N/[N,N]),1)}
\]
such that the composite is injective on $\pi_1$. It follows that the rationalization $K(\Z_p,1)_{\Q} \to K(L_0(N/[N,N]),1)_{\Q}$ of this map is split, hence the composite
\[
\xymatrix{H_*(K(\Z_p,1);\Q) \ar[r] & H_*(K(N,1);\Q) \ar[r] & H_*(K(L_0(N/[N,N]),1);\Q)}
\]
is a split monomorphism, which implies the claim. 
\end{proof}

\begin{prop}\label{prop:dichotomy}
If $X$ is a $p$-complete nilpotent space whose homotopy groups are not all bounded $p$-torsion, then the integral homology groups $H_*(X;\Z)$ and the stable homotopy groups $\pi_*\Sigma^{\infty}X$ both contain an uncountable rational vector space.
\end{prop}
\begin{proof}
Assume that $\pi_*X$ is not all bounded $p$-torsion, and let $\pi_nX$ be the lowest such group. It then follows from \Cref{lem:detection} that $\pi_nX$ contains a class $x$ of infinite order inducing a monomorphism $H_*(K(\Z_p,n);\Q) \to H_*(K(\pi_nX,n);\Q)$. 
Since the map $\tau_{\ge n}X \to X$ is a rational homology equivalence, any rational subgroup of $H_*(\tau_{\ge n}X;\Z)$ must map monomorphically to $H_*(X;\Z)$, so it suffices to prove the homological claim for $\tau_{\ge n}X$. 
The element $x$ yields a map $S_p^n \to \tau_{\ge n}X$ such that the composite $S_p^n \to \tau_{\ge n}X \to K(\pi_nX,n)$ factors as 
\[
\xymatrix{\tau_{\ge n}X \ar[r] & \tau_{\le n}\tau_{\ge n}X \simeq K(\pi_nX,n) \\
S_p^n \ar[r] \ar[u] & \tau_{\le n}S_p^n \simeq K(\Z_p,n). \ar[u]}
\]
It follows from \Cref{lem:rationalclass} and the choice of $x$ that the induced homomorphism in homology 
\[
\xymatrix{H_{2n}(S_p^n;\Z) \ar[r] & H_{2n}(\tau_{\ge n}X;\Z) \ar[r] & H_{2n}(K(\pi_nX,n);\Z)}
\]
maps an uncountable rational vector space monomorphically to $H_{2n}(K(\pi_nX,n);\Z)$, hence so does the map $H_{2n}(S_p^n;\Z) \to H_{2n}(\tau_{\ge n}X;\Z)$. This verifies the claim about the integral homology of $X$. 

Recall that, for any connective spectrum $Y$, the Hurewicz map $\pi_*Y \to H_*(Y;\Z)$ has kernel and cokernel of bounded torsion in each degree. Indeed, the fiber sequence $Y\wedge \tau_{>0}S^0 \to Y \to Y \wedge H\Z$ reduces this claim to showing that $\pi_*(Y\wedge \tau_{>0}S^0)$ is bounded torsion in each degree. This follows from the convergent Atiyah--Hirzebruch spectral sequence
\[
H_s(Y;\pi_t\tau_{>0}S^0) \implies \pi_{s+t}(Y\wedge \tau_{>0}S^0),
\]
because $H_s(Y;\pi_t\tau_{>0}S^0)$ is bounded torsion for all $s$ and $t$. Therefore, any rational vector space in $H_*(Y;\Z)$ may be lifted back to $\pi_*Y$. In particular, an uncountable rational vector space in $H_{2n}(X;\Z)$ may be lifted back to $\pi_{2n}(\Sigma^{\infty}X)$ after suspension.
\end{proof}

\begin{rem}
Suppose $X$ is a $p$-complete nilpotent space such that $\pi_nX$ is the lowest homotopy group not of bounded $p$-torsion. The above argument shows that $H_{2n}(X;\Z)$ contains an uncountable rational vector space. With more work, we can also show that $H_k(X;\Z)$ is derived $p$-complete for $k\le 2n-2$ and thus cannot contain any rational classes. Note that when $X$ is $(n-1)$-connected, this follows immediately from \Cref{cor:stablerange} since $H_k(X;\Z)$ is in the stable range. 
\end{rem}

We can now prove our main theorem. 

\begin{thm}\label{thm:char}
If $X$ is a $p$-complete nilpotent space, then $\Sigma^{\infty}X$ is $p$-complete if and only if $\pi_nX$ is bounded $p$-torsion for each $n$. 
\end{thm}

Note that the torsion exponent of $\pi_nX$ may vary with $n$ and does not need to be bounded uniformly for all $n$. 

\begin{proof}
First assume that $X$ is a $p$-complete nilpotent space with $\pi_nX$ of bounded $p$-torsion for each $n$; we can apply \cite[Ch.~II, 4.7]{bousfield_kan} to see that the Postnikov tower of $X$ can be refined to a tower of principal fibrations whose fibers are Eilenberg--MacLane spaces for bounded $p$-torsion abelian groups. 
The category of bounded $p$-torsion abelian groups forms a Serre class, so Serre theory implies that $H_*(X;\Z) \cong H_*(\Sigma^{\infty}X;\Z)$ is degreewise bounded $p$-torsion. Hence, $\Sigma^{\infty}X$ is $p$-complete as a spectrum by \Cref{cor:homcharpcompletion}.

The converse is a consequence of \Cref{prop:dichotomy}: if $\pi_*X$ is not all bounded torsion, then $H_*(\Sigma^{\infty}X;\Z)$ contains rational classes and thus cannot be derived $p$-complete, hence $\Sigma^{\infty}X$ is not $p$-complete by \Cref{cor:homcharpcompletion}.
\end{proof}

\begin{cor}
If $X$ is a nilpotent space with $H_n(X;\Z)$ and $\pi_nX$ derived $p$-complete for all $n$, then $H_n(X;\Z)$ and $\pi_nX$ are bounded $p$-torsion for all $n$. 
\end{cor}
\begin{proof}
The assumption on $\pi_*X$ implies that $X$ is $p$-complete by \Cref{thm:bkpcompletion}, while the assumption on $H_*(X;\Z)$ shows that $\Sigma^{\infty}X$ is $p$-complete, using \Cref{cor:homcharpcompletion}. It thus follows from \Cref{thm:char} that $\pi_*X$ is degreewise bounded $p$-torsion, hence so is $H_*(X;\Z)$ by the proof of \Cref{thm:char}.
\end{proof}

The analogue of this corollary does not hold stably, as the following example demonstrates.

\begin{ex}
Let $M(\Z_p,n)$ be the Moore space for $\Z_p$ in degree $n\ge 2$. As $H_*(\Sigma^{\infty}M(\Z_p,n);\Z)$ is isomorphic to $\Z_p[n]$, we see that $\Sigma^{\infty}M(\Z_p,n)$ is $p$-complete and consequently has derived $p$-complete stable homotopy groups and integral homology groups. However, $H_n(\Sigma^{\infty}M(\Z_p,n);\Z)\cong \Z_p$ is clearly not bounded $p$-torsion. In particular, $M(\Z_p,n)$ is not $p$-complete, so this also shows that the assumption that $X$ be $p$-complete cannot be dropped in \Cref{thm:char}.
\end{ex}

\section{Rational classes in the stable homotopy groups of $K(\Z_p,n)$}

In this section, we present an example that illustrates how the rational classes in the stable homotopy groups of $p$-complete spaces arise. 
In fact, we present two different approaches: One using the integral homology of $K(\Z_p,n)$, and one using Goodwillie calculus. The latter derivation is entirely stable and might be of independent interest. 

First, we need a well-known auxiliary result; we outline a proof because we were unable to find a published reference for it. For an abelian group $A$ and any $k\ge 0$, let $\Sym_{\Z}^{k}(A)$ and $\Lambda_{\Z}^k(A)$ be the $k$th symmetric power and the $k$th exterior power on $A$, respectively. 

\begin{lem}\label{lem:surprise}
If $k>1$, then $\Lambda_{\Z}^k(\Z_p)$ and the kernel of the multiplication map $\Sym_{\Z}^{k}(\Z_p)\to \Z_p$ are uncountable rational vector spaces.
\end{lem}
\begin{proof}
Since both symmetric and exterior power commute with base-change along $\Z \to \Z/l$ for any prime $l$, the indicated maps are isomorphisms mod $l$. Moreover, $\Sym_{\Z}^k(A)$ and $\Lambda_{\Z}^k(A)$ are torsion-free whenever $A$ is, so both $\ker(\Sym_{\Z}^{k}(\Z_p)\to \Z_p)$ and $\Lambda_{\Z}^k(\Z_p)$ are rational vector spaces. We may therefore base-change to $\Q$, where it is easy to verify that the $\Q$-dimension of the groups under consideration is that of $\Q_p$. 
\end{proof}

\begin{rem}
A similar argument also shows that $\Z_p/\Z_{(p)}$ is a rational vector space with the same $\Q$-dimension as $\Q_p$. 
\end{rem}

\begin{prop}\label{prop:emspaces}
For $n \ge 1$ and all $k > 1$, the stable homotopy group $\pi_{nk}\Sigma^{\infty}K(\Z_p,n)$ contains an uncountable rational vector space. In particular, $\Sigma^{\infty}K(\Z_p,n)$ is not $p$-complete.
\end{prop}
\begin{proof}[First proof]
Let $A$ be an abelian group and recall that $H_*(K(A,n);\Z)$ equipped with the Pontryagin product is a graded commutative algebra such that squares of odd dimensional elements are zero; in fact, it has the structure of a graded divided power algebra, see \cite{em_2,cartan_em} or more recently \cite{richter_divided}. With notation as in the previous lemma, the canonical isomorphism $A \to H_{n}(K(A,n);\Z)$ thus extends to a natural homomorphism
\[
\begin{cases}
\xymatrix{\phi^k(A,n)\colon \Lambda_{\Z}^k(A) \ar[r] & H_{kn}(K(A,n);\Z),} & \text{if } n \text{ odd} \\
\xymatrix{\phi^k(A,n)\colon \Sym_{\Z}^k(A) \ar[r] & H_{kn}(K(A,n);\Z),}& \text{if } n \text{ even}
\end{cases}
\]
for any $n,k>0$. Moreover, we know that $\phi^k(A,n) \otimes_{\Z}\Q$ is a rational isomorphism. It then follows from \Cref{lem:surprise} that, for $k>1$, there exists an uncountable rational vector space which is mapped monomorphically to $H_{kn}(K(\Z_p,n);\Z)$ via $\phi^k(\Z_p,n)$. We thus obtain an uncountable rational vector space in $H_{kn}(K(\Z_p,n);\Z)$ that may be lifted back to give the desired uncountable rational vector space in $\pi_{nk}\Sigma^{\infty}K(\Z_p,n)$ for $k>1$, as in the proof of \Cref{prop:dichotomy}.
\end{proof}

\begin{proof}[Second proof]
We will compute the homotopy groups of $\Sigma^{\infty}K(\Z_p,n) \simeq \Sigma^{\infty}\Omega^{\infty}\Sigma^nH\Z_p$ using Goodwillie calculus~\cite{goodwillie3}. To this end, recall that the Goodwillie tower $(P_k)_{k\ge1}$ associated to the functor $\Sigma^{\infty}\Omega^{\infty}\colon \Sp \to \Sp$ is assembled from fiber sequences of functors
\begin{equation}\label{eq:g1}
\xymatrix{D_k \ar[r] & P_k \ar[r] & P_{k-1}}
\end{equation}
with layers $D_kX \simeq X_{h\Sigma_k}^{\sm k}$, where the homotopy orbits are formed with respect to the permutation action of $\Sigma_k$ (see for example~\cite{kuhnmccarty_omega} and the references given therein). Moreover, the Goodwillie tower $(P_k)_{k\ge0}$ converges for connective spectra, i.e., there is a canonical equivalence
\[
\xymatrix{\Sigma^{\infty}\Omega^{\infty}X \ar[r]^-{\sim} & \lim_kP_kX}
\]
for any connective $X \in \Sp$. We will apply this in the case $X = \Sigma^nH\Z_p$. 

In order to understand the layers, we start by analyzing $\pi_*(\Sigma^nH\Z_p)^{\sm k}$ via the universal coefficient theorem. We claim that, for all $k\ge 1$, the homotopy groups have the following form
\begin{equation}\label{eq:c1}
\pi_*(\Sigma^nH\Z_p)^{\sm k} \cong
\begin{cases}
0 & * < nk \\
\Z_p^{\otimes_{\Z} k} & * = nk \\
\text{finite} & * > nk.
\end{cases}
\end{equation}
By the universal coefficient theorem, we have an isomorphism 
\[
\pi_*(\Sigma^nH\Z_p)^{\sm k} \cong (\pi_*(\Sigma^nH\Z)^{\sm k})\otimes_{\Z} \Z_p^{\otimes k}.
\]
In degrees $*> nk$, the groups $\pi_*(\Sigma^nH\Z)^{\sm k}$ are torsion, so the only torsion-free summand appears in degree $nk$. Since $\pi_*(\Sigma^nH\Z)^{\sm k}$ is finitely generated over $\Z$ in each degree, the claim follows. 

We now plug the formula \eqref{eq:c1} into the convergent homotopy orbit spectral sequence
\[
H_s(\Sigma_k,\pi_t(\Sigma^nH\Z_p)^{\sm k}) \implies \pi_{s+t}D_k(\Sigma^nH\Z_p).
\]
There are two cases: If $t > nk$ or $t<nk$, then the groups $H_s(\Sigma_k,\pi_t(\Sigma^nH\Z_p)^{\sm k})$ are finite or trivial for all $s$, respectively. Let $t=nk$. By \Cref{lem:surprise} and \eqref{eq:c1}, there is an isomorphism $H_s(\Sigma_k,\pi_{nk}(\Sigma^nH\Z_p)^{\sm k}) \cong H_s(\Sigma_k,\Z_p)$ for $s>0$ and $H_0(\Sigma_k,\pi_{nk}(\Sigma^nH\Z_p)^{\sm k})$ contains an uncountable rational vector space $V_k$ if $k>1$. To see the last statement, it suffices to compute the coinvariants on the rational submodule of $\Z_p^{\otimes_{\Z} k}$ by choosing a $\Q$-bases, as in the proof of \Cref{lem:surprise}. Furthermore, since the integral homology of $\Sigma_k$ is finitely generated over $\Z$ in each degree and rationally trivial in positive degrees, $H_s(\Sigma_k,\pi_{nk}(\Sigma^nH\Z_p)^{\sm k})$ is finite for all $s>0$. Combining all this information, we obtain $D_1\Sigma^nH\Z_p\simeq\Sigma^nH\Z_p$ and for $k>1$:
\begin{equation}\label{eq:c2}
\pi_*D_k(\Sigma^nH\Z_p) \cong
\begin{cases}
0 & *<nk \\
V_k \oplus W_k & * = nk \\
\text{finite} & * > nk,
\end{cases}
\end{equation}
where $V_k$ is an uncountable rational vector space and $W_k$ is some abelian group. 

This allows us to derive a structural formula for $\pi_*P_k\Sigma^nH\Z_p$. Consider the following segment of the long exact sequence of homotopy groups associated to the fiber sequence \eqref{eq:g1}:
\[
\xymatrix{\ldots \ar[r] & \pi_{nk+1}P_{k-1}\Sigma^nH\Z_p \ar[r] & \pi_{nk}D_k\Sigma^nH\Z_p \ar[r] &  \pi_{nk}P_k\Sigma^nH\Z_p \ar[r] & \ldots.}
\] 
Because $n\ge 1$, it follows inductively from \eqref{eq:c2} that the term on the left is finite, hence $V_k$ must be a summand in $\pi_{nk}P_k\Sigma^nH\Z_p$. This yields for all $k\ge 1$:
\begin{equation}\label{eq:c3}
\pi_*P_k\Sigma^nH\Z_p \cong
\begin{cases}
0 & *<n \\
V_l \oplus W_l' & * = nl \text{ with } 1\le l \le k \\
\text{finite} & \text{otherwise},
\end{cases}
\end{equation}
where $V_l$ is as above for $l\ge2$, and $V_1$ and $W_l'$ are some abelian groups. 

Finally, since $D_k\Sigma^nH\Z_p$ is $nk$-connective for all $k$, the tower $(\pi_*P_k\Sigma^nH\Z_p)_{k\ge0}$ stabilizes after finally many steps in each degree and hence is Mittag-Leffler. The corresponding Milnor sequence thus degenerates to an isomorphism
\[
\pi_*\Sigma^{\infty}K(\Z_p,n) \cong \pi_*\Sigma^{\infty}\Omega^{\infty}\Sigma^nH\Z_p \cong \lim_k\pi_*P_k\Sigma^nH\Z_p.
\]
Therefore, the claim follows from \eqref{eq:c3}.
\end{proof}

\biblio
\bibliography{duality}\bibliographystyle{alpha}
\end{document}